\documentclass[12pt,twoside]{amsart}
\usepackage{amsxtra}
\usepackage{color}
\usepackage{amsopn}
\usepackage{amsmath,amsthm,amssymb}
\newtheorem{teo}{Theorem}[section]
\newtheorem{prop}{Proposition}[section]
\newtheorem{corol}{Corollary}[section]
\newtheorem{lemm}{Lemma}[section]
\newtheorem{remark}{Remark}[section]
\newtheorem{ex}{Example}[section]

\newcommand{\beq}{\begin{equation}}
\newcommand{\eeq}{\end{equation}}
\newcommand{\bqn}{\begin{eqnarray}}
\newcommand{\eqn}{\end{eqnarray}}
\newcommand{\bqne}{\begin{eqnarray*}}
\newcommand{\eqne}{\end{eqnarray*}}

\newcommand{\C}{{\mathbb C}}

\title[ Astheno-K\"ahler and balanced structures]{Astheno-K\"ahler and balanced structures on  fibrations}

\begin{document}

\author{Anna Fino, Gueo Grantcharov  and Luigi Vezzoni}
\date{\today}
\subjclass[2000]{Primary 53C15 ; Secondary 53C55, 53C30}
\keywords{astheno-K\"ahler, balanced, complex homogeneous space}
\address{Dipartimento di Matematica \lq\lq Giuseppe Peano\rq\rq \\ Universit\`a di Torino\\
Via Carlo Alberto 10\\
10123 Torino\\ Italy}
 \email{annamaria.fino@unito.it, luigi.vezzoni@unito.it}
 \address{ Department of Mathematics and Statistics Florida International University\\
  Miami Florida, 33199, USA}
\email{grantchg@fiu.edu}

\thanks{The work of the first and third authors was supported by the project FIRB ``Geometria differenziale e teoria geometrica delle funzioni'' and by G.N.S.A.G.A. of I.N.d.A.M. The work of the second author is supported by the Simons Foundation grant \#246184}

\maketitle

\begin{abstract} We study the existence of  three classes of Hermitian  metrics on certain types of compact complex manifolds. More precisely, we consider  balanced, SKT and astheno-K\"ahler metrics. We prove that the twistor spaces of compact hyperk\"ahler and negative quaternionic-K\"ahler manifolds do not admit astheno-K\"ahler metrics. Then we provide a construction of astheno-K\"ahler structures on torus bundles over K\"ahler manifolds leading to new examples. In particular, we find  examples of   compact  complex non-K\"ahler manifolds which admit a balanced and an astheno-K\"ahler metrics, thus answering to a question in \cite{STW} (see also \cite{F}). One of these examples is simply connected.
We also show that the Lie groups  $SU(3)$ and $G_2$ admit SKT and astheno-K\"ahler metrics, which are different. Furthermore,
  we investigate  the existence of balanced metrics on compact complex  homogeneous spaces  with an invariant volume form, showing in particular  that if   a compact complex homogeneous space  $M$ with invariant volume admits a balanced metric, then  its first Chern class $c_1(M)$ does not vanish. Finally we characterize Wang C-spaces admitting SKT metrics.

\end{abstract}


\section{Introduction} After it became clear that certain complex manifolds do not admit K\"ahler metrics, the question of finding appropriate generalizations naturally arose. Although a universal type of Hermitian non-K\"ahler metrics has not been found yet, several classes, related to different geometric or physics applications, have been introduced and studied.
The present paper focuses on the existence and interplay between three such classes: astheno-K\"ahler, SKT and balanced metrics
on particular examples of compact complex non-K\"ahler manifolds.

A Hermitian  metric $g$ on a complex manifold $(M,I)$ is called  {\em astheno-K\"ahler} if its fundamental form $F(\cdot,\cdot):=g(I\cdot,\cdot)$ satisfies
$$
dd^cF^{n-2} = 0\,,
$$
where $n$ is the complex dimension of $M$, $d^c=I^{-1}dI$  and $I$ is naturally extended on differential forms. Such metrics were used by Jost and Yau in  \cite{JY} to establish existence of Hermitian harmonic maps which led to some information about the fundamental group of the targets. As application, Carlson and Toledo made use of astheno-K\"ahler metrics  to get restrictions on  the fundamental  group of complex surfaces of class VII \cite{CT}.  Later harmonic maps from more general Hermitian manifolds  have been studied in \cite{LY}.  In \cite{LYZ} Li, Yau, and Zheng conjectured  that compact
non-K\"ahler Hermitian-flat manifolds or similarity Hopf manifolds of
complex dimension   $\geq 3$  do not admit any astheno-K\"ahler metric. Moreover,  Tosatti and  Weinkove proved in \cite{TW}
 Calabi-Yau theorems for Gauduchon and strongly Gauduchon  metrics on the  class of  compact astheno-K\"ahler manifolds.  Until recently  astheno-K\"ahler  metrics were not receiving a lot of attention due to the lack of examples. Indeed, there are not many examples of astheno-K\"ahler manifolds, some of them are given by Calabi-Eckmann manifolds  \cite{Matsuo} and by nilmanifolds   \cite{FT}.  In the present paper we provide a construction which leads to many new ones.

We further recall that a Hermitian metric $g$ on  a complex manifold $(M,I)$ is called  {\em strong K\"ahler with torsion} (SKT)  or {\em pluriclosed},  if its fundamental form $F$ satisfies $dd^cF = 0$, while it is called {\em balanced} if $F$ is co-closed. Moreover, $g$ is called standard (or Gauduchon) if $d d^c F^{n -1} =0$. For $n = 3$ the notions of astheno-K\"ahler and SKT metric coincide, while in higher dimensions they lie in different classes.  The study of SKT metrics was initiated by Bismut in \cite{Bismut} and then it was pursued in many papers (see e.g. \cite{GHR, Strominger, FPS,Cavalcanti} and the references therein).
For  a complex surface the notions of SKT metric and standard metric coincide and in view of \cite{Gauduchon} every  compact complex surface has an SKT metric. In higher dimensions things are different and there are known examples of complex manifolds not admitting SKT metrics.  For instance, Verbitsky showed in  \cite{V}   that the twistor space $M$ of a compact, anti-selfdual Riemannian manifold admits an SKT metric if and only it is K\"ahler (hence if and only if it is isomorphic to $\mathbb{C P}^3$ or to a flag space). This result is obtained by using rational connectedness of twistor spaces, proved by Campana in  \cite{Campana}.

Examples of compact SKT manifolds are provided by principal torus bundles over compact K\"ahler manifolds \cite{GGP}, by nilmanifolds \cite{FPS, Ugarte, EFV}, and other examples can be constructed by using twist construction \cite{Swann} or blow-ups \cite{FT2}.

Balanced geometry is probably the most studied, partly due to its relation with string theory and Strominger's system. The terminology was introduced by  Michelsohn in  \cite{Michelson},
 where balanced metrics were first studied in depth. In particular, Michelsohn showed an obstruction to the existence of balanced metrics by using currents. From Michelsohn's obstruction it follows that Calabi-Eckmann manifolds have no balanced metrics. Balanced metrics are stable under modifications, but not under deformations. Basic examples of balanced manifolds are given by twistor
spaces. First it was known that twistor spaces of selfdual 4-manifolds are balanced \cite{Michelson} and later the same was proven for twistor spaces of a hyperk\"ahler (see \cite{KV}) and quaternionic-K\"ahler  manifolds (\cite{Pontecorvo,AGI}) and, most recently, by the twistor spaces   of a compact hypercomplex manifold \cite{Tomberg}. Taubes in \cite{Taubes} used the twistor space examples to show that every finitely generated group is a fundamental group of a balanced manifold.

It is known that a Hermitian metric on a compact manifold is simultaneously balanced and SKT if and only if it is K\"ahler (see e.g. \cite{AI}). Furthermore, it was proved that many classes of examples of balanced manifolds do not admit any SKT metric  \cite{chiose,chiose2,FV2,FLY,V,FV}.

This led Fino and Vezzoni to ask in \cite{FV2} whether there are compact complex manifolds admitting both SKT and balanced metrics but which admit no K\"ahler structure.
In the same spirit, one can wonder if a compact complex non-K\"ahler manifold could admit both an astheno-K\"ahler and a balanced metric.  This second problem arises from \cite{STW}, where it is proved that the Calabi-Yau type equation introduced in \cite{FWW} is solvable on astheno-K\"ahler manifolds admitting balanced metrics. There was an opinion, formulated as part of a folklore conjecture in \cite{F}, that as in the SKT case, a balanced compact complex manifold can not admit an astheno-K\"ahler, unless it is K\"ahler. Note that, as in the SKT case, a metric on a compact manifold cannot be both balanced and astheno-K\"ahler simultaneously, unless K\"ahler.

In the present paper (Section 4), using the general construction of Section 3,  we construct an example of an $8$-dimensional  non-K\"ahler  principal torus bundle over a torus admitting both an astheno-K\"ahler metric and a balanced metric. It is also a 2-step nilmanifold with holomorphically trivial canonical bundle.  Examples of  nilmanifolds in every dimension $2n\geq 8$ have been constructed independently by Latorre and Ugarte \cite{U}. In Section 5 we also find a simply connected compact example of dimension 22 with non-vanishing first Chern class. It is the complex homogeneous space $SU(5)/T^2$ with a complex structure studied by H. C. Wang \cite{Wa}.

 In Section 2   we show that twistor spaces over compact hyperk\"ahler manifolds and twistor spaces over compact quaternion-K\"ahler manifolds cannot admit  astheno-K\"ahler metrics. We prove  this result  by providing a new obstruction, which generalizes the one in \cite{JY},  to the existence of astheno-K\"ahler metrics by using currents in the spirit of \cite{Michelson}. Furthermore,  in Section 3  we provide a construction of  astheno-K\"ahler metrics on some principal torus bundles over  K\"ahler manifolds generalizing the result of Matsuo  for Calabi-Eckmann manifolds \cite{Matsuo}. The construction could be used to find many new examples. Apart from Section 4, we use it to find astheno-K\"ahler metrics  on  the Lie groups $SU(3)$ and $G_2$. The result gives  examples of  compact complex non-K\"ahler manifolds  admitting both an astheno-K\"ahler and an SKT metric. In  the last part of the
paper  we investigate  the existence of balanced metrics on compact complex manifolds with an invariant volume form. Compact complex homogeneous spaces with invariant volumes have been classified in \cite{Guan}, showing that every compact complex homogeneous space with an invariant volume form is a principal homogeneous complex torus bundle over the product of a projective rational homogeneous space and a complex parallelizable manifold. We obtain a characterisation of the balanced condition in terms of the   characteristic classes  of the associated torus fibration
and of the K\"ahler cone of the projective rational homogeneous space. As a consequence we show that  if   a compact complex homogeneous space  $M$ with invariant volume admits a balanced metric, then  its first Chern class is non-vanishing.

Finally, in the last section we study the existence of SKT metrics on Wang C-spaces, i.e. on compact complex non-K\"ahler manifolds with finite fundamental group, admitting a transitive action by a compact Lie group of biholomorphisms. We show that every Wang C-space admitting an SKT metric can be covered with a product of a compact Lie group and a generalized flag manifold. In particular our example $SU(5)/T^2$ cannot admit SKT metrics.

\section{Non-existence of astheno-K\"ahler metrics on twistor spaces}
In this section we show that  twistor spaces of hyperk\"ahler  and quaternionic-K\"ahler manifolds do not admit any astheno-K\"ahler metric.

A {\em hyperk\"ahler manifold} is a Riemannian manifold $(M,g)$ with holonomy contained in $Sp(n)$. The hyperk\"ahler condition can be characterized by  the existence of three  complex structures $I,J,K$ each one inducing a K\"ahler structure with $g$ and satisfying the quaternionic relation $IJ=-JI=K$. As a consequence, for any $p=(a,b,c)\in S^2, a^2+b^2+c^2=1$, the endomorphism $I_p=aI+bJ+cK$ is an integrable almost complex structure. The twistor space $Tw(M)$ of $M$ is then defined as the space $Tw(M)=M\times S^2$ endowed with the tautological complex structure $\mathcal{I}|_{(x,p)} = I_p|_{T_xM}\times I_{S^2}$. Note that there are  two natural projections $\pi_1: Tw(M)\rightarrow M$ and $\pi_2:Tw(M)\rightarrow S^2$ where the second one is holomorphic, but the first one is not.

A Riemannian manifold $(M,g)$ is called {\em quaternionic-K\"ahler} if its holonomy is in $Sp(n)Sp(1)$. In analogy to the hyperk\"ahler case, the quaternionic-K\"ahler condition can be characterized by the existence of a parallel sub-bundle $D$ of ${\rm End}(TM)$ locally spanned by a triple of almost complex structures $I,J,K$,  each one compatible with $g$ and satisfying the quaternionic relations $IJ=-JI=K$. Although the definition of quaternionic-K\"ahler manifold is similar to the one of hyperk\"ahler manifold, the geometric properties of these two kind of manifolds are rather different. One common feature, however is the existence of a twistor space. The twistor space $Q(M)$ of  a quaternionic-K\"ahler $(M,g)$ is defined as the $S^2$-bundle over $M$ with fiber the  $2$-sphere in $D$ naturally identified with the set of almost complex structures $S = \{ aI+bJ+cK\,\,|\,\,a^2+b^2+c^2=1\}$. Although $I,J,K$ are locally defined, $S$ doesn't depend on their choice.
The twistor space $Q(M)$ has a tautological almost complex structure $\mathcal{I}$, similar to the one defined in the hyperk\"ahler case. It uses a splitting $TQ(M) = H\oplus T_{S^2}$,  where $H$ is a horizontal subspace, defined via the Levi-Civita  connection of $g$  which defines  one  on $D$.  It is known that $\mathcal{I}$ is integrable and the metric $g_Q=\pi^*g|_{H}\oplus tg_{S^2}$ is Hermitian for every positive $t$.
Hyperk\"ahler manifolds are Ricci-flat, while the quaternionic-K\"ahler ones are Einstein. In particular quaternionic-K\"ahler metrics have constant scalar curvature and are called of positive or negative type depending on their  sign.
For a positive quaternionic-K\"ahler manifold $M$, the metric $g_Q$ and the structure $\mathcal{I}$ define a K\"ahler structure for an appropriate choice of $t$. In other words, if $W$ is the fundamental form of $g$ and $\mathcal{I}$, $dW=0$ for some $t$.  In \cite{DDM} one finds a computation of $dd^cW$, from this computation it follows that $dd^cW$ is weakly positive if the scalar curvature of the base is negative.

In order to show that twistor spaces do not admit astheno-K\"ahler metrics, we prove that they are obstructed.
In \cite{JY} Jost and Yau show that on a compact complex manifold admitting an astheno-K\"ahler metric every holomorphic 1-form is closed,  giving an obstruction to the existence of
astheno-K\"ahler metrics. We generalize the obstruction of Jost and Yau in the spirit of \cite{HL}.

Here we recall that a $(p,p)$-{\em current} on a complex manifold $(M,I)$ is an element of the Frechet space dual to the space of $(n-p,n-p)$ complex forms $\Lambda^{n-p,n-p}(M)$. In the compact case, the space of $(p,p)$-currents can be identified with the space of $(p,p)$-forms with distribution coefficients and the duality is given by integration. So for any $(p,p)$-current $T$ and a form $\alpha$
of type $(n-p,n-p)$ we have
$$
\langle T,\alpha\rangle = \int_M T\wedge \alpha.
$$
The operators $d$ and $d^c$ can be extended to $(p,p)$-currents by using the duality induced by the
integration, i.e.,   $dT$ and $d^cT$ are respectively defined via the relations
$$
\langle dT,\beta\rangle = - \int_M T\wedge d\beta, \quad \langle d^cT,\beta\rangle = -\int_MT\wedge d^c\beta.
$$
A $(p,p)$-current $T$ is called {\em weakly positive}
if
$$
i^{n-p}\int_MT\wedge\alpha_1\wedge\overline{\alpha_1}\wedge...\alpha_{n-p}\wedge\overline{\alpha}_{n-p} \geq 0,
$$
 for every (1,0)-forms $\alpha_1,...\alpha_{n-p}$ with inequality being strict for at least one choice of $\alpha_i$'s. The current $T$ is called {\it positive} if the inequality is strict for every non-zero $\alpha_1\wedge\overline{\alpha_1}\wedge...\alpha_{n-p}\wedge\overline{\alpha}_{n-p}$.
From  \cite[Theorem 2.4]{A} (with an obvious change) we have the following:

\begin{teo}
A compact complex manifold admits a positive $dd^c$-closed $(p,p)$-form if and only if it does not admit a $dd^c$-exact weakly positive and non-zero  $(n-p,n-p)$-current.
\end{teo}

Since $F^{n-2}$ is positive, we have the following useful corollary, to which we provide an independent proof.
\begin{corol}\label{cor}
If a compact complex manifold $M$ admits a weakly positive and $dd^c$-exact and non-vanishing $(2,2)$-current, then it does not admit
an astheno-K\"ahler metric.
\end{corol}

\begin{proof} Suppose $dd^c T$ is weakly positive, where we consider $T$ as a form with distribution coefficients. Then for an astheno-K\"ahler metric with fundamental form $F$ we have by integration by parts $$0<\int_M dd^c T\wedge F^{n-2} = \int_M T\wedge dd^c(F^{n-2}) = 0,$$
which gives a contradiction.
\end{proof}

Note that for a holomorphic 1-form $\alpha$, the form $i\partial\alpha\wedge\overline{\partial\alpha}=dd^c(\alpha\wedge\overline{\alpha})$ is weakly positive, which leads to the obstruction of Jost and Yau. Furthermore, we remark that the statement of Corollary  \ref{cor} cannot be reversed, since in general not every positive $(n-2,n-2)$-form arises as $(n-2)$-power of a positive $(1,1)$-form.

\medskip

Then we can prove the non-existence of an astheno-K\"ahler metric on twistor spaces.

\begin{prop} The twistor space $Tw(M)$ of a compact hyperk\"ahler manifold   does not admit any astheno-K\"ahler metric, compatible with the tautological complex structure.

Similarly, the twistor space $Q(M)$ of a compact quaternionic-K\"ahler manifold of negative scalar curvature does not admit an astheno-K\"ahler metric, compatible with the tautological complex structure.
 \end{prop}

\begin{proof}
Let $(M, I, J, K, g)$ be a compact hyperk\"ahler  manifold and let $Tw(M)=M\times S^2$ be its twistor space. Then $G=g_M+g_{S^2}$ gives a Hermitian metric on $Tw(M)$, compatible with the tautological complex structure. We denote by $W$  the fundamental form of $G$. In view of \cite{KV},  given $(p,r)\in Tw(M)$ one has
$$
dd^cW_{(p,r)} = F_p\wedge \omega_{r},
$$
where $F_p(X, Y)=g(I_p X, Y)$ for tangent vectors $X,Y\in T_pM$,  and $\omega$ is the Fubini-Study form on $S^2\equiv \mathbb{CP}^1$. Therefore $dd^cW$ is a weakly-positive non-vanishing $(2,2)$-current  and Corollary \ref{cor} implies that $Tw(M)$ has not astheno-K\"ahler metrics.

About the quaternion-K\"ahler case, using the splitting $TQ(M) = H\oplus TS^2$ one can define the $1$-parameter family of $\mathcal I$-compatible Hermitian metrics $g_Q=\pi^*g|_{H}\oplus tg_{S^2}$, where $t$ is a positive parameter.

Let us denote by $W_Q$ the fundamental of form of $g_Q$.  Then in view of \cite[Theorem 5.4]{DDM}  if the scalar curvature of $g$ is negative, then $dd^cW_Q$ is weakly positive and, consequently, the Corollary \ref{cor} can be applied also in this case and the claim follows.
\end{proof}

\vspace{.2in}

\section{Astheno-K\"ahler metrics on torus bundles}
In   \cite{Matsuo} Matsuo showed the existence of astheno-K\"ahler metrics on Calabi-Eckmann manifolds. Since Calabi-Eckmann manifolds are principal $T^2$-bundles over
$\mathbb{CP}^n\times \mathbb{CP}^m$,  it is quite natural to extend the Matsuo's result to principal torus fibrations over compact K\"ahler manifolds.
\begin{prop}\label{propnoAK}
Let $\pi:P\rightarrow M$ be an $n$-dimensional principal torus bundle over a K\"ahler manifold
$(M,J,F)$ equipped with $2$ connections $1$-forms $\theta_1, \theta_2$ whose curvatures $\omega_1,\omega_2$ are of type $(1,1)$ and are pull-backs from forms $\alpha_1$ and $\alpha_2$ on $M$. Let $I$ be the complex structure on  $P$ defined as the pull-back of $J$ to the horizontal subspaces and as  $I(\theta_1) = \theta_2$ along vertical directions.
Let $\Omega = \pi^*(F)+\theta_1\wedge\theta_2$; then
$$
dd^c\Omega^k=k\,   (\omega_1^2+\omega_2^2)\wedge(\pi^*(F^{k-1}))\,,\quad 1\leq k\leq n-2\,.
$$
In particular if $(\alpha_1^2+\alpha_2^2)\wedge F^{n-3}=0$, then $\Omega$ is astheno-K\"ahler and  if a torus bundle with $2$-dimensional fiber over a K\"ahler base admits an  SKT metric, then  it is astheno-K\"ahler.
\end{prop}
\begin{proof}
Since $d\theta_i=\omega_i$  and $dF=0$, we have
$$
d\Omega = \pi^*(dF) + \omega_1\wedge \theta_2-\theta_1\wedge\omega_2 = \omega_1\wedge \theta_2-\theta_1\wedge\omega_2,
$$
which implies
$$
d^c\Omega =\omega_1\wedge \theta_1+\theta_2\wedge\omega_2\,.
$$
Therefore
$$
dd^c\Omega = \omega_1^2+\omega_2^2
$$
and
$$
d\Omega\wedge d^c\Omega = -(\omega_1^2+\omega_2^2)\wedge\theta_1\wedge\theta_2.
$$
This  leads to the following simplified expression for $dd^c\Omega^{n-2}$
$$
\begin{array}{lcl}
dd^c\Omega^{k} &=&k\,d(d^c\Omega\wedge\Omega^{k-1})\\[3pt]
&=&  k\, (dd^c\Omega\wedge\Omega+(k-1)d\Omega\wedge d^c\Omega)\wedge\Omega^{k-2}\\[3pt]
&=& k\,   (\omega_1^2+\omega_2^2)\wedge(\pi^*(F)+\theta_1\wedge\theta_2-(k-1)\theta_1\wedge\theta_2)\wedge\Omega^{k-2}\\[3pt]
&=&k\,   (\omega_1^2+\omega_2^2)\wedge(\pi^*(F)-(k-2)\theta_1\wedge\theta_2)\wedge\Omega^{k-2}.
\end{array}
$$
Moreover we have
$$
\begin{aligned}
\Omega^{k-2} = (\pi^*(F)+\theta_1\wedge\theta_2)^{k-2} = \pi^*\left(F^{k-2}\right)+(k-2)\theta_1\wedge\theta_2\wedge \pi^*(F^{k-3})\\
= \left((\pi^*(F)+(k-2)\theta_1\wedge\theta_2)\right)\wedge \pi^*(F^{k-3}),
\end{aligned}
$$
so after substitution we get
$$
dd^c\Omega^{k} = k\,   (\omega_1^2+\omega_2^2)\wedge(\pi^*(F^{k-1}))\,,
$$
as required.
%
%
\end{proof}
\begin{remark}
{\em The situation is similar to the ones in  the examples of nilmanifolds obtained in \cite{FT}.}
{\em Proposition \ref{propnoAK} includes the case  of   Calabi-Eckmannn manifolds studied in \cite{Matsuo}. In the Calabi-Eckmann case we have   $\omega_1 = \Phi_1+a\Phi_2$ and $\omega_2=b\Phi_2$, where $\Phi_1$ and $\Phi_2$ are the Fubini-Study forms on the factors. Note that  the metric  in  \cite{Matsuo}  is not SKT. In fact to construct examples of astheno-K\"ahler structures we only need to find closed (1,1)-forms $\omega_1$ and $\omega_2$ such that $\omega_1^2\wedge F^{n-2}$ and $\omega_2^2\wedge F^{n-2}$ are constant multiples of the volume form with  different signs. This is achieved by taking appropriate linear combinations of $\theta_1$ and $\theta_2$.}
\end{remark}

\section{Interplay between special classes of Hermitian metrics}
A problem in complex non-K\"ahler geometry is to establish if the existence of two Hermitian metrics  belonging two different classes  on a  compact complex manifold  impose some restrictions.

In \cite{FV2} Fino and Vezzoni  conjectured that the existence of a balanced and an SKT metric on a compact complex manifold $(M,I)$, forces $M$ to be K\"ahler. Analogue problems make sense by replacing the SKT and the balanced assumption with other different classes of Hermitian metrics.

In this section we provide examples of compact manifolds admitting both SKT and astheno-K\"ahler metrics and an  example  having  both balanced and  astheno-K\"ahler metrics.

Basic examples of SKT manifolds are given by the  compact Lie groups  endowed with the Killing (biinvariant) metric and any of the compatible  Samelson's complex structures \cite{Samelson}. Samelson's construction depends on a choice of the maximal torus and such complex  structures are compatible if they are compatible with the metric restricted to this torus. Let $G$ be a simple compact Lie group and $T^n$ be one of its maximal toral subgroups, so that $Fl = G/T^n$ is a flag manifold. Assume that $n$ is even. Then the projection $\pi:G \rightarrow Fl$ (called Tits fibration) is as above and is holomorphic for the Samelson's complex structures. However the metric induced on $Fl$ from the Killing metric on $G$ is not K\"ahler. The K\"ahler ones arise from forms on adjoint orbits (see \cite{Besse}). We use the relation between the two metrics in the following proposition.

\begin{prop}
A compact semisimple Lie group of  even dimension $2n > 6$  and of   rank two endowed with its Samelson's complex structure admits an astheno-K\"ahler non-SKT metric while its canonical SKT metric is not astheno-K\"ahler. In particular, the Lie groups  $SU(3)$ and $G_2$ admit astheno-K\"ahler metrics.
\end{prop}
\begin{proof}
Let $G$ be a compact Lie group of real dimension $2n$ and let $T^2$ be the maximal torus of $G$.

 Let $\pi:G\rightarrow G/T^2=Fl$ be the holomorphic projection onto the corresponding flag manifold as mentioned above. Let $\alpha_1$ and $\alpha_2$ be invariant representatives of the characteristic classes of $\pi$ and $g_1$ be the bi-invariant metric on $G$. We suppose that the metric $g_1$  is compatible with the complex structure  on $G$ so that if $\Omega_1$ is the fundamental form, it is well-known that $dd^c\Omega_1=0$. Moreover $g_1$ induces a naturally-reductive metric on $Fl$ with form $F_1$ and such that $\Omega_1 = F_1+\theta_1\wedge \theta_2$,  where $\theta_1$ and $\theta_2$ are orthonormal invariant 1-forms on $G$ with respect to $g_1$ and $I(\theta_1)=\theta_2$. Then $d\theta_i=\omega_i$ for $i=1,2$, where $\omega_i$ are independent linear combinations of $\alpha_i$. We can always achieve this, since there are invariant 1-forms $\overline{\theta}_i$ with $d\overline{\theta}_i = \alpha_i$. This is an example of the torus bundle construction from the previous section and we have
$$
0=dd^c\Omega_1 = \pi^*(dd^cF_1)+\omega_1^2+\omega_2^2\,
$$
i.e.,
\begin{equation}\label{eqnalpha}
\omega_1^2+\omega_2^2 = -dd^c(\pi^*(F_1))\,.
\end{equation}

\medskip
Now we consider the K\"ahler-Einstein metric $g_2$ on $Fl$ and denote by $F_2$ its fundamental form. Let $\Omega_2$ be the Hermitian form on $G$ defined as
$$
\Omega_2=\pi^*(F_2)+\theta_1\wedge\theta_2
$$
We next show that  $dd^c\Omega_2^{n-2}=0$.
 By equation \eqref{AK}, $\Omega_2^{n-2}$ is $dd^c$-closed if and only if
$$
(\omega_1^2+\omega_2^2)\wedge F_2^{n-3}=0\,.
$$
But \eqref{eqnalpha} implies
$$
(\omega_1^2+\omega_2^2)\wedge F_2^{n-3}=-dd^cF_1\wedge F_2^{n-3}=-d\left(d^cF_1\wedge F_2^{n-3}\right)\,.
$$
Hence $(\omega_1^2+\omega_2^2)\wedge F_2^{n-3}$ is exact on $Fl$. It is also invariant, so it is constant multiple of the volume form. Since its integral is zero  it vanishes which implies that $\Omega_{2}^{n-2}$ is $dd^c$-closed.

Since $SU(3)$ and $G_2$ are even-dimensional compact (semi)simple Lie groups of rank two, the proposition is proved.
\end{proof}

Note that the SKT form $\Omega_1$ is not astheno-K\"ahler and $\Omega_2$ is not SKT. Metrics which are both SKT and astheno-K\"ahler are constructed in \cite{FT}.

\medskip

\begin{ex}  {\rm Next we provide an example of a compact $8$-dimensional  complex manifold admitting a balanced and an astheno-K\"ahler metric.  The example is  the total space of a principal  $T^2$-bundle over a $6$-dimensional torus.  Let $\pi: M \rightarrow T^6$  be  the principal  $T^2$  bundle   over $T^6$  with characteristic classes
$$
a_1=dz_1\wedge d \overline{z}_1+dz_2\wedge d \overline{z}_2-2 dz_3\wedge d \overline{z}_3\,,\quad a_2 = dz_2 \wedge d\overline{z}_2- dz_3\wedge d\overline{z}_3,
$$
where $(z_1, z_2, z_3)$ are complex coordinates on $T^6$.
Consider on $T^6$ the standard complex structure and let
$$
F_1 =  dz_1\wedge d\overline{z}_1+dz_2\wedge d\overline{z}_2+ dz_3\wedge d\overline{z}_3\,,\quad F_2 =  dz_1\wedge d\overline{z}_1+dz_2\wedge d\overline{z}_2+ 5 dz_3\wedge d\overline{z}_3\,,
$$
then the $a_i$'s are traceless with respect to $F_1$ and $(a_1^2+a_2^2)\wedge F_2=0$.  Let $\theta^j$ be connection $1$-forms  such that $d \theta^j  = \pi^*  a_j$ and define
$$
\omega_1=\pi^* F_1+  \theta^1  \wedge \theta^2, \quad \omega_2=\pi^* F_2+  \theta^1  \wedge \theta^2\,.
$$
The $2$-forms $\omega_1$  and $\omega_2$ define respectively a balanced metric $g_1$ and an astheno-K\"ahler metric $g_2$ on $M$  compatible with the  integrable complex structure so that the projection map $\pi$  is holomorphic.  $M$ can be alternatively described as the $2$-step nilmanifold $G/\Gamma$,  where $G$ is the $2$-step nilpotent Lie group with structure equations
$$
\left \{ \begin{array}{l}
d e^j =0, \quad j = 1, \ldots, 6,\\[3pt]
d e^7 = e^1 \wedge e^2 + e^3 \wedge e^4 - 2 e^5 \wedge e^6,\\[3pt]
d e^8 = e^3 \wedge e^4 -  e^5 \wedge e^6
\end{array}
\right.
$$
and $\Gamma$ is a co-compact discrete subgroup, endowed with the invariant  complex structure $I$  such that $I e_1 = e_2$, $I e_3 = e_4, I e_5 = e_6, I e_7 = e_8$.  In this setting
the $2$-form
$ e^1 \wedge e^2 + e^3 \wedge e^4 +  e^5 \wedge e^6 +e^7 \wedge e^8 $
 defines a  balanced metric   and  $e^1 \wedge e^2 + e^3 \wedge e^4 + 5  e^5 \wedge e^6 + e^7 \wedge e^8$ gives an astheno-K\"ahler metric.}
\end{ex}

\begin{remark}
{\em The example above has holomorphically trivial canonical bundle and is not simply connected. Other examples on nilmanifolds in every dimension greater or equal to 8 are found by Latorre and Ugarte \cite{U}. More examples could be found on torus bundles with base a K\"ahler manifold which has known cohomology ring. In the next section we provide such example on  a torus bundle over a flag manifold. }
\end{remark}

\section{Balanced metrics on compact complex homogeneous spaces with invariant volume}

A $2n$-dimensional manifold  $M$  is a {\em complex homogeneous space with invariant volume} if there is a complex structure and a nonzero $2n$-form on $M$ both preserved by a transitive Lie transformation group.

Compact complex homogeneous spaces with invariant volumes have been classified in \cite{Guan}, showing that every compact complex homogeneous space with an invariant volume form is a principal homogeneous complex torus bundle over the product of a projective rational homogeneous space and a complex parallelizable manifold.

A {\em rational homogeneous manifold} $Q$ (also called {\em a generalized  flag manifold}) is by definition a  compact complex manifold that can be realized as a closed orbit of a linear algebraic group in some projective space. Equivalently, $Q = S/P$ where $S$ is a complex semisimple Lie group and $P$ is a parabolic subgroup, i.e., a subgroup of $S$ that contains a maximal connected solvable subgroup (i.e. a Borel subgroup).
A {\em complex parallelizable} manifold is a compact quotient of a complex Lie group by a discrete subgroup.

Chronologically, Matsushima first considered the special case of a semisimple group action, proving in \cite{Matsushima} that
if $G/H$ is a compact complex homogeneous space with a $G$-invariant volume, then $G/H$  is a holomorphic fiber bundle over a rational homogeneous space and the fiber is a complex reductive parallelizable manifold as a fiber. This kind of fibration is called {\it Tits fibration}.

Matsushima's result \cite{Matsushima} was improved by Guan in \cite{Guan}

\begin{teo}
Every compact complex homogeneous space $M$ with an invariant volume form is a principal homogeneous complex torus bundle
$$
\pi\colon M\rightarrow G/K\times D
$$
over the product of a projective rational homogeneous space and a complex parallelizable manifold.
\end{teo}

In particular, when $M$ is invariant under a complex semisimple Lie group the Tits fibration is a torus fibration.

Furthermore, in \cite{Guan} it is shown that the
bundle $\pi: M\rightarrow G/K\times D$ arises as a factor of the product of two principal complex torus
bundles. One is $\pi_1: G/H\rightarrow G/K$, which is the Tits fibration for
$G/H$ with fiber tori, and the other is $\pi_2:D_1\rightarrow D$, where $D_1$ is again compact
complex parallelizable and the fiber is a complex torus, which is in the center
of $D_1$. The action for the factor bundle is the anti-diagonal one. The projection $M\rightarrow D$ is the Tits fibration with complex parallelizable fibers.

The structure of $\pi_1: G/H\rightarrow G/K$ is discussed in \cite{Wa}. The compact complex homogeneous spaces $G/H$ with finite fundamental group and $G$ compact are called by Wang \cite{Wa}  {\em $C-$spaces.} They fall in two classes K\"ahlerian C-spaces and non-K\"ahlerian C-spaces. The K\"ahlerian C-spaces are well studied and are precisely the generalized flag manifolds or in case $H$ and $K$ are connected, $H=K$.
More generally, by \cite{Guan2}  a compact  complex homogeneous space does not admit a symplectic structure unless it is a product of a flag manifold and a complex parallelizable manifold.  In the sequel we'll use  the terminology {\em Wang's C-spaces}  for  non-K\"ahlerian C-spaces.

The  characteristic classes of $\pi: M\rightarrow G/K\times D$ are
$(\omega_1+\alpha_1, \omega_2+\alpha_2,..., \omega_{2k}+\alpha_{2k})$, where $(\omega_1,...,\omega_{2k})$ are the
characteristic classes of $G/H\rightarrow G/K$ which are (1,1) and $(\alpha_1,...,\alpha_{2k})$ are the
characteristic classes of $D_1\rightarrow D$. Using averaging one can show that there is a unique $G$-invariant representative in each class $\omega_k$. The second fibration is a complex torus fibration and its
(complex) characteristic classes are of type (2,0) with respect to the complex structure on the base $D$. In
particular $\alpha_i$ are of type (2,0)+(0,2), i.e. they have representatives of this type.

We start the characterization of the balanced condition with the following observation:

\begin{lemm}\label{L1}  Assume there is an invariant Hermitian metric on the generalized flag manifold with respect to which all traces of the characteristic class $\omega_i$ vanish. Then $M$ admits a balanced metric.
\end{lemm}

\begin{proof} We know that $M$ is the total space of  the principal torus fiber construction $\pi: M\rightarrow G/K\times D$.  Suppose that $g = g_1 + g_2$  is an invariant  Hermitian metric on the base manifold $G/K\times D$  with  a fundamental form
$F = F_1 + F_2$.  In view \cite[Theorem 2.2]{AG} every complex parallelizable manifold has a balanced metric and we may assume $g_2$ balanced. Consider the Hermitian metric $g_M$ on $M$ defined by
\begin{equation}\label{toric1}
g_M := \pi^*  g+ \sum_{l=1}^{2k}  (\theta_l \otimes \theta_l),
\end{equation}
where $\theta_i$ are connection 1-forms  with  $J \theta_{2j -1} = \theta_{2j}$. Then the fundamental form of the metric $g_M$  is given by
$$
F_M = \pi^* F + \sum_{l=1}^{k}  (\theta_{2l  -1} \wedge \theta_{2l}).
$$
From \cite[formula (4)]{GGP} the co-differential of the fundamental form $F_M$  on $M$  is given by
$$\delta F_M =\pi^*(\delta F) + \sum _{i = 1}^{2k} g(F,\omega_i+\alpha_i)\theta_i =  \pi^*(\delta F) + \sum_{i =1}^{2k}  g(F,\omega_i)\theta_i$$
Since $\alpha_i$ are of type (2,0) and (0,2), their traces with respect to any (1,1) form vanish.   From the formula above,
if there is a balanced metric on the generalized flag manifold  $G/K$ such that all traces of the classes $\omega_i$ vanish, then $M$ admits a balanced metric. The claim follows since every invariant Hermitian metric on a generalized flag manifold is balanced (see Lemma \ref{LO} below).
\end{proof}

By  \cite[Theorem 8.9]{GW}  on a reductive homogeneous almost complex manifold $G/K$ such that
the isotropy representation of $K$  has no invariant $1$-dimensional
subspaces, the fundamental form of every invariant almost Hermitian metric is co-closed. This holds, for example, if the isotropy representation is irreducible or if $G$ and $K$ are reductive Lie groups of equal rank. Moreover, Theorem 4.5 in  \cite{GW}  gives a criterion to establish if a reductive homogeneous almost Hermitian manifold $G/K$ is Hermitian. Therefore, applying  applying Theorem 8.9 of \cite{GW}  to the generalized flag manifolds  we have
\begin{lemm} \label{LO}
Every invariant Hermitian metric on a generalized flag manifold  $G/K$ is balanced.

\end{lemm}

The main result of this section is:

\begin{teo} \label{maintheorem}
Let $M$ be a complex compact homogeneous space admitting an invariant volume form and $\pi: M \rightarrow G/K$ be its Tits fibration with $G/K$ rational homogeneous and $\pi_1: G/H\rightarrow G/K$ be the associated torus fibration of the Guan's representation of $M$, with characteristic classes $\omega_1,...,\omega_k$.  Then $M$ admits a balanced metric if and only if the span of $\omega_1,...,\omega_k$ does not intersect the closure of the K\"ahler cone of $G/K$.
\end{teo}

Notice that the  first Chern class of a generalized flag manifold always belongs to the K\"ahler cone and it vanishes if and only if the first Chern class of the base is in the span of $ \omega_i.$  Hence, taking into account \cite{Grantcharov},  we have:

\begin{corol}
If $M$ admits a balanced metric, then $c_1(M) \neq 0$. In particular compact semisimple Lie groups   do not admit any balanced metric compatible  with Samelson's complex structure.
\end{corol}

For the proof of the Theorem we need some facts about the algebraic structure of the generalized flag manifolds and their cones of invariant Hermitian and K\"ahler metrics.

Let $G$ be a compact semisimple Lie group and $H$ a closed subgroup such that $G/H$ is a complex homogeneous space. Let $\mathfrak{g}$ and $\mathfrak{h}$
be the corresponding Lie algebras and $\mathfrak{g}^c$, $\mathfrak{h}^c$ their complexifications.  As a complex manifold $G/H$ has a useful form of complexification $G^c / H^c$ where $G^c$ and $H^c$ are complex Lie
groups with Lie algebras ${\mathfrak g}^c$ and ${\mathfrak h}^c$ and $G$ is a compact real form of $G^c$, while
$H=H^c\cap G$. By a result of Wang \cite{Wa}, there is an inclusion ${\mathfrak h}^c_{ss} \subset
{\mathfrak h}^c \subset
 {\mathfrak k}^c$, where ${\mathfrak k}^c$ is a parabolic subalgebra, which
is a centralizer of a torus and ${\mathfrak k}^c_{ss} = {\mathfrak h}^c_{ss}$. Here the subscript
 \lq\lq $ss $\rq\rq denotes the semisimple part. In particular

$$ {\mathfrak h}^c = {\mathfrak a} + {\mathfrak h}^c_{ss}$$
where $\mathfrak a$ is a commutative subalgebra of some Cartan subalgebra $\mathfrak t$ of ${\mathfrak g}^c$, which is also the maximal toral subalgebra of ${\mathfrak k}^c$.  The
parabolic algebra ${\mathfrak k}^c$ is ${\mathfrak k}^c = {\mathfrak t} + {\mathfrak k}^c_{ss}$ and is equal to
the normalizer of ${\mathfrak h}^c$ in ${\mathfrak g}^c$. The sum here is not direct because part of ${\mathfrak
t}$ is contained in ${\mathfrak k}^c_{ss}$. Let $K^c$ be a parabolic subgroup if $G^c$ with algebra ${\mathfrak
k}^c$ and $G^c/K^c$ is the corresponding generalized flag manifold. If $K=G\cap K^c$, then the induced map $\pi_1: G/H\rightarrow G/K$ is the fibration from above. Fix
a system of roots $R \in {\mathfrak t}^*$ defined by $\mathfrak t$ in ${\mathfrak g}^c$. There is also a
 distinguished set of simple roots $\Pi$
 in $R$ which forms a basis for ${\mathfrak t}^*$ as a (complex) vector
 space and defines a splitting $R=R^+\bigcup R^-$ of $R$ into positive and negative roots.

 Now every invariant complex structure on the generalized flag manifold is determined by an ordering of the
system of roots of ${\mathfrak g}^c$.  The complex structure on $G/K$
defines a subset $\Pi_0$ in $\Pi$ which
corresponds to ${\mathfrak k}^c$. This correspondence determines the second cohomology of $G/K$ and we provide some details about it.
  In general ${\mathfrak k}^c_{ss}$ is determined
by the span of all roots $R_0$ in $R$ which are positive with respect to
$\Pi_0$. Then the complement $\Pi - \Pi_0 = \Pi'$ provides a basis
for
 the center $\zeta$ of ${\mathfrak k}^c$ and there is an
identification $span_{\bf Z}(\Pi') = H^2(G/K,\mathbb{Z})$. The
identification (see for example \cite{Al}) is:

$$ \xi \rightarrow \frac{i}{2\pi} d\xi,$$
where $\xi$ is considered as a left invariant 1-form on $G$ which is a subgroup of $G^c$ and $d\xi$ is
$ad({\mathfrak k})$-invariant, hence defines a 2-form on $G/K$. This form is obviously closed and in fact
defines non-zero element in $H^2(G/K,\mathbb{Z})$. Moreover every class in $H^2(G/K,\mathbb{Z})$ has unique
representative of this form. It is known that the K\"ahler cone can be identified with a positive Weyl chamber determined by the order in $\Pi'$.

For the description of the invariant Hermitian metrics on $G/K$ we need first the (reductive) decomposition
$${\mathfrak g}^c = {\mathfrak k}^c \oplus {\mathfrak m}^c,$$
where ${\mathfrak m}^c = {\mathfrak m}^+\oplus {\mathfrak m}^-$ and ${\mathfrak m}^{\pm} = \sum_{\alpha\in R'^{\pm}} {\mathfrak g}_{\alpha}$. Here
$R'^+$ and $R'^-$ are the positive and negative roots in $R'=R - R_0$ and ${\mathfrak g}_{\alpha}$ are the corresponding root spaces. The complex structure on $G/K$ is given by $I|_{{\mathfrak m}^{\pm}} = {\pm}i Id$. Now let ${\mathfrak m} = \sum {\mathfrak m}_i$ be the decomposition of ${\mathfrak m}$ into irreducible components under the action of ${\mathfrak k}^c$. Let $B$ be the metric given by the negative of the Cartan-Killing form. Any invariant Hermitian metric on $G/K$ is given by positive numbers $\lambda_i>0$ as follows: $$g= \sum_i \lambda_i B_i,$$ where $B_i = B|_{{\mathfrak m}_i}$. Its fundamental form is $F = \sum_i \lambda_iF_i$ with $F_i=B_i \circ I$.

It can be written also as
\begin{equation}\label{toric2}
F = \sum_{\alpha\in{\mathfrak m}^+} -i\lambda_{\alpha}E_{\alpha}^*\wedge E_{-\alpha}^*
\end{equation}
 where $E_{\alpha}$ are unit (and necessary orthogonal) vectors in ${\mathfrak g}_{\alpha}$ and $E_{\alpha}^*$ their duals. If $\alpha, \beta \in {\mathfrak m}_i$, then $\lambda_{\alpha}=\lambda_{\beta}$. Then the  condition $d F= 0$ is equivalent (see \cite{AD}) to the condition:

For every $\alpha, \beta \in {\mathfrak m}^+$ with $\alpha+\beta \in {\mathfrak m}^+$, $\lambda_{\alpha}+\lambda_{\beta} = \lambda_{\alpha+\beta}.$

We can summarize the  previous observations as:

\begin{lemm}\label{L2}
The cone of the invariant Hermitian metrics (Hermitian cone for short) on  the generalized flag manifold $G/K$  is identified with the first octant in the space of all invariant $(1,1)$ forms by assigning to each metric its fundamental form. The cone of the invariant K\"ahler metrics on  the generalized flag manifold $G/K$ is obtained by intersecting the Hermitian cone with the linear subspace of all closed invariant $(1,1)$-forms.

\end{lemm}

Now we prove Theorem \ref{maintheorem}.

\begin{proof}[Proof of Theorem $\ref{maintheorem}$]Suppose that the span $C$ of the characteristic classes $\omega_i$ does not intersect the K\"ahler cone. Without loss of generality assume that $\omega_i$ are invariant. Then clearly $C$ does not intersect the cone of invariant Hermitian metrics either. From \ref{toric2} this cone is given by the "first octant" - all $\lambda_{\alpha}$'s are positive. Clearly a hyperplane does not intersect its closure if and only if its normal vector is in it. Then from Lemma \ref{L2} there is an element $F'=\sum_{\alpha\in{\mathfrak m}^+} -i\lambda_{\alpha}E_{\alpha}^*\wedge E_{-\alpha}^*$ of the Hermitian cone, which is orthogonal with respect to a metric of type (\ref{toric2}) above with all $\lambda_{\alpha}=1$, to all $\omega_i$. This is equivalent to all $\omega_i$ being traceless with respect to $F=\sum_{\alpha\in{\mathfrak m}^+} -i\sqrt{\lambda_{\alpha}}E_{\alpha}^*\wedge E_{-\alpha}^*$. This $F$ is a fundamental form of a balanced metric $g$ on $G/K$ by Lemma \ref{LO}, which by the torus bundle construction gives rise to a balanced metric on $G/H$. Then by Lemma \ref{L1} we have a balanced metric on $M$.

In the opposite direction, if there is such a class, it defines a non-negative and non-zero form $\alpha$ which pulls back to an exact form on $G/H$. We can also see that its pullback to $M$ is exact too. But  as in \cite{Michelson}, $0 = \int_M \pi^*(\alpha)\wedge F^{n-1}>0$ for any positive $F$ on $M$ with $dF^{n-1}=0$. So such $F$ does not exist.
\end{proof}

Below we provide an example of a complex compact homogeneous Wang's C-space carrying both balanced and astheno-K\"ahler metrics.
In particular this gives an example of compact simply connected non-K\"ahler complex manifold admitting balanced and astheno-K\"ahler metrics, but no SKT metric. More precisely we have:

\begin{prop}
The homogeneous space $SU(5)/T^2$ for appropriate action of $T^2$ is simply connected and has an invariant complex structure which admits both balanced and astheno-K\"ahler metrics, but  doesn't admit any SKT metric.
\end{prop}

\begin{proof}
Consider the flag manifold $SU(5)/T^4$. Then a reductive decomposition for  this homogeneous space  is $\frak{su}(5) = {\frak t}  \oplus {\frak m}$,
where $\frak t$ is the subspace of traceless diagonal matrices with imaginary entries and $\frak m$ is the subspace of skew-Hermitian matrices with vanishing diagonal entries. After complexification, a standard choice of simple roots is the following: let $a_{ij}$ be the basis of  $\frak{ gl}(5, \C)^* $ dual to the standard one in $\frak{gl}(5,\C)$ containing matrices with 1 at $(i,j)$th place and all other entries 0. Then a set of positive roots for $\frak{sl}(5, \C)$ is $a_{ij}$ for $i<j$ and $e_{i,i+1} = a_{ii}-a_{i+1i+1}$,  for $1\leq i\leq 4$. According to the standard theory (see e.g. \cite{BH}),  the forms $de_{i,i+1}$, $1\geq i \geq 4$  form a basis of $H^2(SU(5)/T^4)$.

 Moreover, the complex structure on ${\frak m}^{\C}$ is given by $I(a_{ij})=ia_{ij}$ for $i<j$ and $I(a_{ij})=-ia_{ij}$ for $i>j$. Denote by $\alpha_{ij}$ the $2$-form $\frac{i}{2}a_{ij}\wedge a_{ji}$. For $i<j$ these are (1,1)-forms and the K\"ahler form for the bi-invariant metric on $SU(5)/T^{4}$ is  given by $\sum_{i<j} \alpha_{ij}$.

Note   that $[A_{ij},A_{ji}]=A_{ii}-A_{jj}$, so $$de_{i,i+1}(A_{jk},A_{kj})= -  e_{i,i+1}([A_{jk},A_{kj}]) =  -(\delta_i^j  -\delta_i^k) + (\delta_{i+1}^j-\delta_{i+1}^k).$$
Therefore, by computing $d e_{i, i+1}$, $1 \geq i \geq 4$,  we get that   the following four 2-forms:
$$
\begin{array}{l}
\omega_1 = 2\alpha_{12}+\alpha_{13}+\alpha_{14}+\alpha_{15}-\alpha_{23}-\alpha_{24}-\alpha_{25},\\[3pt]
\omega_2= -\alpha_{12}+2\alpha_{23}+\alpha_{24}+\alpha_{25}-\alpha_{34}-\alpha_{35}+\alpha_{13},\\[3pt]
\omega_3 = -\alpha_{13}-\alpha_{23}+2 \alpha_{34}+\alpha_{14}+\alpha_{24}+\alpha_{35}-\alpha_{45},\\[3pt]
\omega_4 = \alpha_{15}+\alpha_{25}+\alpha_{35}+2\alpha_{45}-\alpha_{14}-\alpha_{24}-\alpha_{34}
\end{array}
$$
form a basis of $H^2(SU(5)/T^4)$.

We  have  that $\omega_1+\omega_2+\omega_3+\omega_4$ is  weakly positive definite with 3 zero directions and  $3\omega_1+5\omega_2+6\omega_3+6 \omega_4$ is strictly positive. Now with respect to the bi-invariant metric (which is  balanced),  the traces of $\omega_i$ are all equal to   2. In view of Proposition \ref{propnoAK} in order to show the existence of an  astheno-K\"ahler  metric on $SU(5)/T^2$ we need to find two traceless classes. If we consider   $F_1=\omega_1+\omega_2-\omega_3-\omega_4$,  $F_2=3\omega_1-\omega_2-\omega_3-\omega_4$ and  the strictly  positive $2$-form $\Omega=3\omega_1+5\omega_2+6\omega_3+6 \omega_4 + 10 (\omega_1+\omega_2+\omega_3+\omega_4)$, we get  that $F_1^2\wedge \Omega^8 >0$ and $F_2^2\wedge\Omega^8 <0$. It is then sufficient  to change  either $F_1$ or $F_2$ by a constant to have them satisfying the condition $(F_1^2+F_2^2)\wedge\Omega^8=0$. Furthermore $SU(5)/T^2$ has a balanced metric (here we can for instance apply Theorem \ref{maintheorem}).  Moreover, it does not admit any K\"ahler structure.

Note that the forms  $\frac{1}{2\pi i}F_1$ and $\frac{1}{2\pi i}F_2$ from above define integer classes and a lot of information about the topology of the space $SU(5)/T^2$ as a principal torus bundle over the flag $SU(5)/T^4$ can be obtained. In particular it is simply connected and with non-vanishing first Chern class.

 Using the obstruction found by Cavalcanti  \cite[Theorem 5.16]{Cavalcanti} we can now show that $SU(5)/T^2$ does not admit any  SKT metric, since it can  not  have   symplectic forms. Indeed, it is possible to prove that $h^{3,0}  (SU(5)/T^2)$ and  $h^{2,1} (SU(5)/T^2)$ both vanish. To calculate the Hodge numbers we use a refinement of the Borel's spectral sequence in \cite{Tanre} which provides an explicit model for the Dolbeault cohomology of principal torus bundles. Recall that a Dolbeault model of a compact complex manifold $M$ is  a morphism $\phi: (\Lambda V, \delta) \rightarrow (\Omega^c(M),\overline{\partial})$ from a commutative differential bi-graded algebra $V=\oplus_{p,q}V^{p,q}, \delta(V^{p,q})\subset V^{p,q+1}$,  to the Dolbeaut complex of $M$ which preserves the grading and induces an  isomorphism on the cohomology. Then Proposition 8 of  \cite{Tanre} describes the Dolbeaut model of a torus bundle $M$ over a K\"ahler base $B$ with $H^2(B, \mathbb{C})\equiv H^{1,1}(B)$ in terms of the de Rham model of $B$. It follows that the  Dolbeault cohomology $H^{3,0} (SU(5)/T^2)$ vanishes, because the flag  manifold $SU(5)/T^4$ has no (2,0)-cohomology. Also any element $ \alpha \in H^{2,1}$ has the form $\alpha = \beta\wedge \omega$ with $\beta$ being the (1,0)-form   and $\overline{\partial}\beta = F_1+iF_2$ and $\omega = \sum_i a_i\gamma_i$. Then $\overline{\partial}\alpha = (F_1+iF_2)\wedge (\sum_i a_i\gamma_i)$, for some generator $\gamma_i$ such that   $\sum_i\gamma_i^2+\sum_{i<j}\gamma_i\wedge\gamma_j = 0$. We can check that the quadratic form $\sum_ix_i^2+\sum_{i<j}x_ix_j$ has maximal rank. However $(F_1+iF_2)\wedge (\sum_i a_i\gamma_i)$ corresponds to a quadratic form of lower rank. So they can not be proportional. This shows that $h^{2,1}=h^{3,0}=0$. On the other side, every class in $H^2(SU(5)/T^2)$ is a pullback of a class on the base flag manifold, so it can not have a maximal rank. Hence $SU(5)/T^2$ doesn't admit a symplectic structure (this is consistent with \cite{Guan2}).
\end{proof}

No general result is known for the existence of  SKT metrics on complex homogeneous spaces.  Complex parallelizable manifolds cannot admit SKT metrics, by using the   argument as in \cite{DLV}  for the invariant case and the symmetrization process in \cite{Belgun,FG,Ugarte}. Indeed,    the existence   of an SKT metric, implies the existence of a   unitary coframe of  invariant  $(1,0)$-forms  $\{ \zeta_i \} $ such that  $\overline \partial  \zeta_i =0,$ so we can suppose that the fundamental form of the  SKT metric  is given by  $\frac {i}{2} \sum_i \zeta_i \wedge \overline \zeta_i$,  which  cannot be $d d^c$-closed.  Moreover, by \cite{Matsuo} a complex parallelizable  manifold cannot  admit any astheno-K\"ahler metric (unless the space is a complex torus).
In the next section, we generalize the arguments above so as to
characterize SKT Wang C-spaces.


\section{SKT metrics on   non-K\"ahler C-spaces}

Recall that a Wang C-space (or non-K\"ahler C-space) is a compact complex manifold admitting a transitive action by a compact Lie group of biholomorphisms and finite fundamental group. According to Wang \cite{Wa}, such a space  admits a transitive action of a compact semisimple Lie group. The aim of this section is to prove the following

\begin{teo}\label{SKTWangCspaces}
Every SKT Wang C-space is (up to a finite cover)  the product of a compact Lie group and a generalized flag manifold.
\end{teo}

Before giving the proof  of the theorem we need some preliminary lemmas.

\begin{lemm}\label{h30-vanish}
Let $M=G/H$ be a Wang C-space. Then  $h^{3,0} (M)=0.$
\end{lemm}

\begin{proof} The Hodge numbers  of $M$ can be computed by using the Tanre model \cite{Tanre} for the Dolbeault cohomology of principal torus bundles. As in the proof of Proposition 5.1 we use the Tits fibration  $G/H \rightarrow  G/K$ and the fact that $h^{2,0} (G/K)$  vanishes. Indeed, by  \cite[14.10]{BH}  $G/K$ is a
rational projective algebraic manifold over $\mathbb C$ all of whose cohomology is of Hodge type $(p, p).$ Now the Lemma follows from \cite{Tanre}, Proposition 8.
\end{proof}

\begin{lemm}
 Let $M=G/H$ be a Wang C-space with $G$ simple. Then $h^{2,1}  (M)= 0 $ unless $H$ is discrete.
\end{lemm}
\begin{proof} We first consider the case $H$ abelian. In this case $H$ is contained in a maximal torus $T$ and the Tits fibration is $G/H\rightarrow G/T$.  From a classical result  in \cite{Borel2} the cohomology  ring $H^*(G/T, {\mathbb C})$  is  generated by the products of $\omega_i$'s. The only relations among them  are given by $Q_i(\omega_1,...,\omega_n) =0$, where $Q_i[x_1,\cdots,x_n]$ are all polynomials invariant under the Weyl group $W_G$ of $G$ acting on $\mathbb{R}[x_1,\cdots,x_n]$ and the product in $H^*(G/T, {\mathbb C})$  is the wedge product of the corresponding representatives. According to a result by Chevalley \cite{Chevalley}, for a simple Lie group $G$,  there exists up to a constant  only one  $W_G$-invariant quadratic polynomial $p\in {\mathbb C}[x_1,\ldots,x_n]$; $p$ is the polynomial corresponding to the Killing form of the Lie algebra of $G$.
 In particular, $H^4 (G/T, {\mathbb C})$ is isomorphic to the space of homogeneous quadratic polynomials factored by $p$. Moreover $p$ is  negative definite over real numbers and, consequently,  it has a maximal rank over complex numbers.

Assume now that  $h^{2,1}(G/H)\neq 0$ and let $\omega_1,\omega_2,...\omega_n$ be the generators of the space $H^{1,1}(G/T, {\mathbb C})$.  From Tanre's model  it follows that there exists  a quadratic relation  involving the $\omega_i$. Indeed,  elements of $H^{2,1}(G/H, {\mathbb C})$ are of the form $\sum_i \alpha_i\wedge\omega_i$, where $\alpha_i$ are some vertical $(1,0)$-forms.   Let $Q(\omega_1,...,\omega_n) = \sum  q_{ij} \omega_i\omega_j$ as above,  where $Q$ is a quadratic polynomial in ${\mathbb C} [x_1,\ldots, x_n]$.

Since up to a constant  there exists only one  $W_G$-invariant quadratic polynomial $p$ in ${\mathbb C}[x_1,\ldots,x_n]$,  $Q=c p$ for a constant $c$ and $Q$ has a maximal rank too. But if the Tits fibration has positive-dimensional fiber, then $Q$ depends on $y_1,..,y_m$ - variables with $m<n$, where $y_i$'s are linear functions of $x_i$'s. In particular  the  diagonal  form of $Q$  has at most $m$ non-zero entries and  it can not have a maximal rank in ${\mathbb C}[x_1,...,x_n]$. So the Tits fibration has  a discrete fiber and the statement follows.
\vspace{.1in}

Now assume $H$ arbitrary.  Then the base of the Tits fibration $\pi\colon G/H\rightarrow G/K$ is a generalized flag manifold and $K$ contains a maximal torus $T$. A  result  by Bernstein-Gelfand-Gelfand in \cite{BGG} implies that $\pi^*: H^*(G/K, {\mathbb C} )\rightarrow H^*(G/T, {\mathbb C})$ is injective. Moreover, the image of $\pi^*$ consists in   the $W_K$-invariant elements of $H^*(G/T, {\mathbb C})$. So,  if $\omega_1,...,\omega_k$ are the characteristic classes of the Tits fibration, then $\pi^*(\omega_i)$ belongs to $H^{1,1}(G/T,{\mathbb C})$ and if $Q=\sum q_{ij} \omega_i\omega_j = 0$, then $\pi^*(Q)=0$ in $H^{1,1}(G/T,{\mathbb C})$. But then there will be additional quadratic  relation among the set of the $W_K$-invariant polynomials in ${\mathbb C}[x_1,...x_n]$. When $W_K$ is non-trivial this is impossible. So, $H$ has to be  either abelian or discrete and, from the first part of the proof, it has to be discrete.
\end{proof}

\begin{lemm}\label{h21-vanish}
 Let $M=G/H$ be a Wang C-space as above, but with $G$ semisimple. Then $h^{2,1}(M)\neq 0$ only if $M$ has a finite cover which is  biholomorphic to a product of a compact even-dimensional Lie group and another Wang C-space.

\end{lemm}

\begin{proof} As in the previous proof, the assumption leads to  the existence of a quadratic relation $Q$ on the characteristic classes of the Tits fibration $\pi\colon G/H \rightarrow G/K$. This time the base $G/K$  is a generalized flag manifold   with $G$ a  semisimple Lie group, so by \cite[Corollary, p. 1148]{Borel} it is a product of generalized flag manifolds $G_1/K_1\times..\times G_k/K_k$ with $G_i$ simple.
Now we identify  the elements of $H^4(G/K, {\mathbb C})$ with the $W_K$-invariant quadratic polynomials in ${\mathbb C}[x^{(1)}_1,..,x^{(1)}_{s_1},  \ldots, x^{(k)}_1,..,x^{(k)}_{s_k}]$,  where $x^{(i)}_1,..,x^{(i)}_{s_i}$ correspond to  the generators of $H^{1,1}(G_i/T_i, {\mathbb C})$ and $T_i$ is the maximal tori in $K_i$. This follows because the maximal torus of $K$ is the maximal torus of $G$ and it  is the product $T = T_1\times \cdots \times T_k$ of maximal tori $T_i$ of $K_i$, which are also maximal in $G_i$. Then the polynomial corresponding to $Q$ is of the type $Q =c_1p_1+...+c_k p_k$ for  some constants $c_i$,  where $p_i$ is the unique (up to a constant) quadratic $W_{K_i}$-invariant polynomial depending on the variables corresponding to the classes of $H^{1,1}(G_i/T_i,{\mathbb C})$. Suppose now that $c_1\neq 0, c_2= \ldots =c_k=0$. Then $Q$ is a function of only  the variables $x^{(1)}_1, \ldots,x^{(1)}_{s_1}$ and has a maximal rank.  As a consequence  the Tits fibration has as fiber the torus  $T_1$ and characteristic classes only in $H^{1,1}(G_1/T_1, {\mathbb C})$ so $K_1$ is abelian and since $T_1$ is maximal abelian,  $K_1 = T_1$. Moreover as in the previous Lemma, we have $(1,0)$-forms $\alpha_1,\alpha_2,...,\alpha_{s_1}$ such that $\overline{\partial}(\sum \alpha_i\wedge \omega^1_i) = Q(\omega^1_1,...,\omega^1_{s_1})$. In particular the real and imaginary parts of $\alpha^1_i$ are all in $\mathfrak{t}_1^*$,  where $\mathfrak{t}_1$ is the Lie algebra of $T_1$. The the restriction of the Tits bundle over $G_1/K_1=G_1/T_1$ is (up to a finite cover) $G_1\rightarrow G_1/T_1$ and the tangent bundle of $G_1$ is complex-invariant, so $M$ is biholomorphic (up to a finite cover) to  the product $G_1\times (G_2\times...\times G_k)/\overline{H}$,  where $\overline{H}$ contains $K_2\times...\times K_k$.

Suppose now that more than one of $c_i$'s is non-zero. Then $Q$ depends only on the variables, corresponding to the nonvanishing $c_i$ and has maximal rank in them. So again, using the same considerations as above, $M$  is up to a finite cover the product of $G_i$'s for these $i$ and some Wang C-space invariant under the action of the product of $G_j$'s corresponding to the $c_j$'s with $c_j=0$.
\end{proof}

\begin{proof}[Proof of Theorem $\ref{SKTWangCspaces}$]
Every Wang C-space $M$ is represented as $M=G/H$ for a compact semisimple Lie group $G$. Now assume that $M$ admits an SKT metric but it  has no K\"ahler metrics. From a result by Borel \cite{Borel} (see also \cite{ZB}, Theorem 5.8 ) $M$ does not admit an invariant symplectic form. By averaging, $M$ does not admit any symplectc structure. Then by  Cavalcanti's result \cite{Cavalcanti}
about the non-existence of SKT and non-K\"ahler metrics on a compact complex manifold  and Lemma \ref{h30-vanish}  we have $h^{2,1}(M) \neq 0$.
By Lemma \ref{h21-vanish}, $M$ is biholomorphic (up to a finite cover) to a product $G_1\times M_1$, where $M_1$ is another Wang C-space. We can embed $M_1$ into (possible finite cover of ) $M$ as a complex submanifold and note that a restriction of an SKT metric to a complex submanifold is again SKT, we obtain an SKT metric on $M_1$. If $M_1$ admits no K\"ahler metric again, we can continue until we get a factor which a C-space, but  it admits a K\"ahler metric, or is empty.

\end{proof}

\begin{remark}{\em
A product of  a K\"ahler  space and an SKT space is SKT. Since the generalized flag manifolds admit K\"ahler structures (and  they are the only homogeneous manifolds of compact semi-simple Lie groups which do, by \cite{Borel}), the spaces of Theorem 6.1 admit SKT metrics.}

{\em Note also that on the product of $G_i$ for the $i$ with $c_i\neq 0$ in the proof, the complex structure does not have to be a product of complex structures  on $G_i$. This is the case of $S^3\times S^3$ for example.}
\end{remark}

\smallskip

{\bf Acknowledgments.} The work on this project started with the second named author's visit to the University of  Torino. He is grateful to Mathematics Department there, as well as Max Plank Institute (Bonn) and Institute of Mathematics and Informatics of the Bulgarian Academy of Sciences for the hospitality and support during the preparation of parts of this paper.  We would like to   thank Valentino  Tosatti for suggesting the  problem  of the existence of  an astheno-Kahler metric on a compact balanced non-K\"ahler  complex manifold. Moreover,  we are grateful to    Liana David and Luis Ugarte for useful comments on the paper.
We would like also to thank the anonymous referees for useful
comments.

 \end{document}